\documentclass{article} 

 \usepackage[usenames,dvipsnames]{pstricks}
 \usepackage{epsfig}
 \usepackage{pst-grad} % For gradients
 \usepackage{pst-plot} % For axes
 \usepackage[space]{grffile} % For spaces in paths
 \usepackage{etoolbox} % For spaces in paths
 \makeatletter % For spaces in paths
 \patchcmd\Gread@eps{\@inputcheck#1 }{\@inputcheck"#1"\relax}{}{}
 \makeatother
\usepackage{comment}

\usepackage{extarrows}

\usepackage[usenames,dvipsnames]{pstricks}
\usepackage{epsfig}
\usepackage{pst-grad} % For gradients
\usepackage{pst-plot} % For axes
\usepackage[space]{grffile} % For spaces in paths
\usepackage{etoolbox} % For spaces in paths
\makeatletter % For spaces in paths
\patchcmd\Gread@eps{\@inputcheck#1 }{\@inputcheck"#1"\relax}{}{}
\makeatother

\usepackage[overload]{empheq}

 \usepackage[usenames,dvipsnames]{pstricks}
 \usepackage{mathtools}
 \usepackage{epsfig}
 \usepackage{pst-grad} % For gradients
 \usepackage{pst-plot} % For axes
 \usepackage[space]{grffile} % For spaces in paths
 \usepackage{etoolbox} % For spaces in paths
 \makeatletter % For spaces in paths
 \patchcmd\Gread@eps{\@inputcheck#1 }{\@inputcheck"#1"\relax}{}{}
 \makeatother

\usepackage{amsmath}
\usepackage{amssymb}
\usepackage{amsthm}
\usepackage[latin1]{inputenc}
\usepackage{graphicx}
\usepackage[usenames,dvipsnames]{pstricks}
\usepackage{epsfig}
\usepackage{pst-grad} % For gradients
\usepackage{pst-plot} % For axes
\usepackage{ifthen}
\usepackage{color}

\newcommand{\intav}[1]{\mathchoice {\mathop{\vrule width 6pt height 3 pt depth  -2.5pt
\kern -8pt \intop}\nolimits_{\kern -6pt#1}} {\mathop{\vrule width
5pt height 3  pt depth -2.6pt \kern -6pt \intop}\nolimits_{#1}}
{\mathop{\vrule width 5pt height 3 pt depth -2.6pt \kern -6pt
\intop}\nolimits_{#1}} {\mathop{\vrule width 5pt height 3 pt depth
-2.6pt \kern -6pt \intop}\nolimits_{#1}}}

\def\polhk#1{\setbox0=\hbox{#1}{\ooalign{\hidewidth\lower1.5ex\hbox{`}\hidewidth\crcr\unhbox0}}}

\newcommand{\co}{\operatorname{co}}

\newcommand{\tr}{\operatorname{Tr}}

\newcommand{\bd}{{\bf d}}

\newtheorem{teo}{Theorem}

\newtheorem{Definition}{Definition}
\newtheorem{Lemma}{Lemma}
\newtheorem{Corollary}{Corollary}
\newtheorem{Proposition}{Proposition}
\newtheorem{Remark}{Remark}
\newtheorem{Assumption}{A}

\usepackage{setspace}
\begin{document}

%\usepackage[displaymath, mathlines, pagewise]{lineno}
%\linenumbers

\title{Stationary fully nonlinear mean-field games}
\author{P\^edra D. S. Andrade,\;\; Edgard A. Pimentel}
\date{\today}

\maketitle

\begin{abstract}
\begin{spacing}{1.15}
\noindent In this paper we examine fully nonlinear mean-field games associated with a minimization problem. The variational setting is driven by a functional depending on its argument through its Hessian matrix. We work under fairly natural conditions and establish improved (sharp) regularity for the solutions in Sobolev spaces. Then, we prove the existence of minimizers for the variational problem and the existence of solutions to the mean-field games system. We also investigate a unidimensional example and unveil new information on the explicit solutions. Our findings can be generalized to a larger class of operators, yielding information on a broader range of examples. 
\end{spacing}

\medskip


\medskip 

\end{abstract}

\vspace{.1in}
	
\section{Introduction}\label{sec_intro}

In the present work we examine fully nonlinear mean-field games (MFG, for short) in close connection with a minimization problem driven by a Hessian-dependent Lagrangian. We consider the system
\begin{spacing}{1.0} 
\begin{equation}\label{eq_main}
	\begin{cases}
		F(D^2u)\,=\,m^{1/(p-1)}&\;\;\;\;\;\mbox{in}\;\;\;\;\;B_1\\
		\left(F_{ij}(D^2u)\,m\right)_{x_ix_j}\,=\,0&\;\;\;\;\;\mbox{in}\;\;\;\;\;B_1,
	\end{cases}
\end{equation}
\end{spacing}
\vspace{.09in}
\noindent where $F$ is a fully nonlinear $(\lambda,\Lambda)$-elliptic operator and $p\geq 2$. In \eqref{eq_main}, $F_{ij}(M)$ stands for the derivative of $F(M)$ with respect to the entry $m_{ij}$ of the matrix $M=(m_{ij})_{i,j=1}^d$. We refer to \eqref{eq_main} as \emph{mean-field game} since the system has an \emph{adjoint structure}. That is to say that the second equation is the formal adjoint of the linearization of the first. 

The MFG system in \eqref{eq_main} is obtained as the first compact variation of
\[
	I[u]\,:=\,\int_{B_1}\left[F(D^2u)\right]^p\,\bd x,
\]
where $I$ is defined over an appropriate class of functions. Therefore, we can understand the system in \eqref{eq_main} as the Euler-Lagrange equation associated with the minimization problem
%We note the second equation in \eqref{eq_main} is the formal adjoint of the linearization of the first one, in the $L^2$-sense. For that reason, we refer to \eqref{eq_main} as \emph{fully nonlinear mean-field game}. The unknown in \eqref{eq_main} is the pair $(u,m)$.
%
%The variational counterpart of \eqref{eq_main} is the problem
\begin{equation}\label{eq_mainvar}
	\int_{B_1}\,\left[F(D^2u)\right]^{p}\,\bd x\;\;\longrightarrow \;\;\min.
\end{equation}
% In fact, the first compact variation of \eqref{eq_mainvar} yields \eqref{eq_main} as the associated Euler-Lagrange equation, at least formally.

The contribution of the present work is to explore the interplay between \eqref{eq_main} and \eqref{eq_mainvar} and to establish new results covering the regularity theory and the well-posedness for the MFG system, as well as a characterization of the minimization problem in \eqref{eq_main}. 

In particular, we establish improved (sharp) regularity for the solutions to \eqref{eq_main} in Sobolev spaces. Then, under general conditions on $F$ -- which include convexity -- we prove the existence of minimizers for \eqref{eq_mainvar}. Then, a further argument yields the existence of solutions to the associated MFG system. We also consider the case of non-convex operators $F$; here, a relaxation argument produces information on the minimum of the associated energy. Finally, we study a one-dimensional toy-model and unveil distinctive properties of the problem \eqref{eq_main}-\eqref{eq_mainvar}. Although we have chosen to present our results in the context of toy-model \eqref{eq_main}, our techniques account for more general formulations, i.e., including lower order terms. The prototypical gradient-dependence we could include is of the form
\begin{equation}\label{eq_chatopracaralho}
	H(p,x)\,\sim\,A(x)\left(1\,+\,\left|p\right|^2\right)^\frac{q}{2}\,+\,V(x),
\end{equation}
where $A\in L^\infty(B_1)$ and $V$ is a merely bounded and measurable potential.
 
The equations comprising \eqref{eq_main} are a fully nonlinear PDE and an elliptic equation in the double-divergence form. Interesting on their very own merits and leading foundational developments in the profession, those equations have been largely studied in the course of the last fifty years. Clearly, to put together a meaningful list of references on their regard is out of the scope of this paper. Therefore, we refrain from mentioning any further work than the monographs  \cite{ccbook} and \cite{bkrsbook}.

The fully nonlinear equation in \eqref{eq_main} can be regarded as the Hamilton-Jacobi associated with an (stochastic) optimal control problem. Of particular interest in the nonlinear dependence on the Hessian of the value function. Notice also the introduction of the associated density in the underlying cost functional, through the mapping $z\mapsto z^\frac{1}{p-1}$. Finally, it is worthy noticing that the choice for this dependence is two-fold. First, it appears naturally in the variational derivation of \eqref{eq_main}. Moreover, it describes a cost functional that penalizes crowds.

The second equation in \eqref{eq_main} is a Fokker-Planck whose coefficients depend on the value function $u$. In fact, it describes the distribution associated with the stochastic process with infinitesimal generator
\[
	F_{ij}(D^2u)\frac{\partial^2}{\partial x_i\partial x_j}.
\]
This description is useful in framing the fully nonlinear model \eqref{eq_main} into the context of the toy-models appearing in the MFG literature.

The mean-field games theory was introduced in \cite{ll1,ll2,ll3} as a mathematical framework to model scenarios of strategic interactions involving a (very) large number of players. Its mathematical formulation is completely described by the so-called \emph{master equation}. Under additional assumptions, on the independence of the underlying stochastic process, models simplify substantially. Here, the work-horse of the theory is the coupling of a Hamilton-Jacobi and a Fokker-Planck equation.  In this context, several authors advanced the topic in a variety of directions. The existence (and uniqueness) of solutions -- both in the elliptic and parabolic settings -- is the object of \cite{cgbt,cd2,CLLP,cllp13,cargra,LCDF,pim4,pim2,pim3,marcir1,marcir2,marcir3,marcir4}, whereas numerical developments are reported is \cite{achdou2013finite,MR2928376,CDY,DY}, among others. Applications of MFG theory to life and social sciences can be found, for instance in \cite{moll,bmoll1,bmoll2,bmoll3}. Finally, the analysis of the master equation has been advanced in \cite{LCDF,carpor,benone,bentwo,cardel1,delaruegalera,gangboswiech}. See also the monographs \cite{cardaliaguet,bensoussan,pim1}.

Since MFG theory embeds in the context of (stochastic) optimal control problems, it is relevant to make sure that the formulation in \eqref{eq_main} is in line with this framework. In this regard, we notice that Bellman operators are typically of class $\mathcal{C}^{0,1}$, but fail to be $\mathcal{C}^1$-regular. Nevertheless, to make sense of the second equation in \eqref{eq_main}, the Lipschitz regularity implied by the $(\lambda,\Lambda)$-uniform ellipticity is enough.

A distinctive feature associated with mean-field games systems concerns gains of regularity for the solutions, \textit{vis-a-vis} the same equations taken isolated. We investigate the occurrence of a similar phenomenon in the case of \eqref{eq_main}. In fact, we prove that a solution $(u,m)$ is such that $D^2u$ has better integrability than in the case where $u$ solves $F(D^2u)=\mu$ with $\mu\in L^1(B_1)$ taken arbitrarily. This findings are reported in Proposition \ref{prop_gains}. For values of $p\geq 2$ in a suitable range, this result yields a counterpart in H\"older spaces.

Although substantial advances have been produced by the profession, fully nonlinear formulations are yet to be addressed in the literature. In fact, in \cite{gomes_ferreira} the authors put forward a well-posedness analysis based on monotonicity properties and the Minty-Browder machinery; see \cite[Chapter 5]{evans_cbms}. Under certain regularity assumptions on $F$, together with monotonicity conditions, they claim that solutions to \eqref{eq_main} would be available; see \cite[Section 7.1]{gomes_ferreira}.

%Typically, the heuristic derivation of a MFG problem involves a population with initial density $m_0$, evolving according to an stochastic differential equation 
%\begin{equation}\label{eq_mfg1}
%	\bd x_t\,=\,b_t\bd t\,+\,\left(\sqrt{2}Id\right)\,\bd W_t,
%\end{equation}
%where $b_t\in\mathbb{R}^d$ is a control and $W_t$ is a $d$-dimensional standard Brownian motion. Also, the agents in this population minimize a functional of the form
%\begin{equation}\label{eq_mfg2}
%	J(x,t)\,:=\,\int_{t}^TL(x_s,b_s)\,+\,g(m)\bd s\,+\,u_T(x_T),
%\end{equation}
%where $L:\mathbb{R}^{2d}\to\mathbb{R}$ is a Lagrangian, $u_T:\mathbb{R}^d\to\mathbb{R}$ is a given function and $g$ encodes the dependence of objective function on the density $m$. Let the Hamiltonian $H=H(x,p)$ be given by the Legendre transform of $L=L(x,b)$.It is well known that the value function $u(x,t)$ and the density $m(x,t)$ associated with \eqref{eq_mfg1}-\eqref{eq_mfg2} satisfy the coupling
%\begin{spacing}{1.0} 
%\begin{equation}\label{eq_MFG}
%	\begin{cases}
%		&-u_t\,+\,H(x,Du)\,=\,\Delta u\,+\,g(m)\\
%		&m_t\,-\,\div\left(D_pH(x,Du)m\right)\,=\,\Delta m.
%	\end{cases}
%\end{equation}
%\end{spacing}
%
%\medskip
%
%A similar reasoning accounts for the derivation of the elliptic counterpart of \eqref{eq_MFG}. Also, we notice that the function $g$ plays a pivotal role in the properties of the system. Typical choices for this term are $g(z):=z^\alpha$ or $g(z):=\ln m$; singular formulations and the the interplay between the focusing and defocusing cases have also been studied. See {\bf CITE MANY PEOPLE}.

Our findings advance the MFG theory by establishing the existence of weak solutions to \eqref{eq_main} with no regularity assumptions on $F$. Indeed, we work under ellipticity and ellipticity-like conditions on the operator governing the MFG system. This is the content of Theorem \ref{teo_eulerlagrange}.

Since the analysis of \eqref{eq_main} interweaves with properties of \eqref{eq_mainvar}, we proceed with some context on the latter. An important aspect concerning the functional in \eqref{eq_mainvar} is its dependence on the Hessian of the argument function $u$. The most elementary example is the case $F(M):=\tr(M)$, which leads to the functional governing the equilibrium of thin plates
\begin{equation}\label{eq_plates}
	I_\Delta[u]\,:=\,\int_{B_1}\left|\tr(D^2u)\right|^2\,\bd x;
\end{equation}
the Euler-Lagrange equation associated with this functional is the plate equation
\[
	\Delta\Delta u\,=\,0\;\;\;\;\;\mbox{in}\;\;\;\;\;B_1,
\]
also known as biharmonic operator.

Functionals of the form \eqref{eq_plates} are relevant in the context of conformally invariant energies. In fact, \eqref{eq_plates} is conformally invariant in dimension four. This fact suggests the development of a regularity program for biharmonic mappings in line with the theory available in the harmonic setting. In \cite{achang1} the authors develop this theory for biharmonic mappings from (domains in) $\mathbb{R}^d$ into $m$-dimensional spheres $\mathbb{S}^m$. See also \cite{achang2}.

More general classes of Hessian-dependent functionals can be found in various contexts. First, they represent an important strategy in by-passing the lack of convexity in minimization problems. Consider, for example, the non-convex functional
\begin{equation}\label{eq_nonconvexfuncintreg}
	J[u]\,:=\,\int_{B_1}\left(|Du|^2\,-\,1\right)^2\bd x.
\end{equation}
An alternative to regularize $J$ is to consider 
\begin{equation}\label{eq_nonconvexfuncint}
	J_\varepsilon[u]\,:=\,\int_{B_1}\left(|Du|^2\,-\,1\right)^2\,+\,\varepsilon^2|D^2u|^2\bd x;
\end{equation}
since $J_\varepsilon$ is convex with respect to the terms of higher order, it is possible to investigate the existence of a minimizer $u_\varepsilon$. Ideally, information on \eqref{eq_nonconvexfuncintreg} would be recovered through \eqref{eq_nonconvexfuncint}, by taking the singular limit $\varepsilon\to 0$. In some cases, such a limit entails further complexities; these are known as \emph{microstructures}. We mention that $J_\varepsilon$ is referred to as Aviles-Giga functional; see \cite{avilesgiga1} and \cite{avilesgiga2}; see also \cite{kohnicm}.

Functionals depending on the Hessian of a given function also appear in the context of energy-driven pattern formation and nonlinear elasticity, in the study of the mechanics of solids. One example regards the study of wrinkles appearing in a twisted ribbon \cite{kohn2}. In this setting, the energy functional depends on the thickness $h$ of the ribbon, which is regarded as a parameter as in
\[
	\int_{B_1}|M(u,v)|^2\,+\,h^2|B(u,v)|^2\bd x,
\]
where $M$ and $B$ are symmetric tensors accounting for the stretching and bending energies of the system, respectively. Although $M$ depends on $u$ and $v$ only through lower order terms, $B$ depends on $v$ through its Hessian, namely
\[
	B(u,v)\,\sim\,\|D^2u\|\,+\,C.
\]
The case of interest is the limit $h\to 0$. A further instance where higher order functionals appear in the context of solid mechanics concerns the formation of blister patterns in thin films on compliant substrates \cite{kohn3}. As before, the functional depends on lower order terms and a small (convex) perturbation driven by the Hessian of the minimizers. 

In this literature, it is relevant to obtain \emph{matching} upper and lower bounds for the functional. It means that both upper and lower bounds scale accordingly with respect to small parameters. In the case of \cite{kohn2}, for instance, the small parameter is the thickness of the ribbon, $h$. For recent developments in this literature we refer the reader to \cite{venkataramani}, \cite{maggi1}, \cite{maggi2} and the references therein.

When examining \eqref{eq_mainvar}, our focus is two-fold. Firstly, we work under the convexity of $F$ and prove the existence of a minimizer, Then we notice that such critical point is indeed a weak (distributional) solution for the associated Euler-Lagrange equation. This fact produces the existence of solutions to \eqref{eq_main}.

Our second approach to \eqref{eq_mainvar} drops the convexity of the operator. Here, a relaxation argument meets an ellipticity pass-through mechanism for the convex envelope of $F$. This mechanism closely relates to coercivity. As a consequence, we are able to characterize the minimum of the energy governed by the fully nonlinear operator.

To complete this work, we study a class of unidimensional problems admitting explicit solutions. Our endeavours here are inspired by the analysis in \cite{dgomesuni}. With this respect, we unveil interesting aspects of the problem. For example, we notice that minimizing solutions pair affine function with uniform distributions, regardless of the growth regime of the functional; see Theorem \ref{teo_unidimensional}.

The remainder of this paper is organized as follows: Section \ref{subsec_ma} puts forward our main assumptions, whereas Section \ref{subsec_pnr} gathers a few definitions and preliminary results. In Section \ref{sec_gofreg} we prove Proposition \ref{prop_gains} describing the improved regularity of the solutions to \eqref{eq_main}. The existence of minimizers for \eqref{eq_mainvar} and the well-posedness for \eqref{eq_main} is the subject of Section \ref{sec_convex}. In Section \ref{sec_relaxation} we drop the convexity assumption on $F$ and resort to relaxation arguments to characterize the minimizers of \eqref{eq_mainvar}. The analysis of explicit solution appears in Section \ref{sec_unid} and closes the paper.

\bigskip

\noindent{\bf Acknowledgements} The authors are grateful to Marco Cirant, Alessandro Goffi and Diogo A. Gomes for their interest and insightful comments on the material in this paper. PA is partially funded by CAPES-Brazil and FAPERJ-Brazil. EP is partially funded by CNPq-Brazil (Grants \# 433623/2018-7 and \# 307500/2017-9), FAPERJ-Brazil (Grant \# E26/200.002/2018), ICTP-Trieste and Instituto Serrapilheira (Grant \# 1811-25904). This study was financed in part by the Coordena\c{c}\~ao de Aperfei\c{c}oamento de Pessoal de N\'ivel Superior - Brazil (CAPES) - Finance Code 001.

\section{Main assumptions and preliminary results}\label{sec_mapr}

Our findings depend on (natural) conditions imposed on the elliptic operator $F$, as well as on a few auxiliary results. This section puts forward an account of our main assumptions and gathers former results used in the paper.

\subsection{Main assumptions}\label{subsec_ma}

We start by imposing an ellipticity condition on the operator $F$. In what follows, $\mathcal{S}(d)$ denotes the space of real symmetric matrices; it will be identified with $\mathbb{R}^\frac{d(d+1)}{2}$, whenever convenient. 

\begin{Assumption}[$(\lambda,\Lambda)$-uniform ellipticity]\label{assump_F}
We suppose the operator $F:\mathcal{S}(d)\to\mathbb{R}$ to be $(\lambda,\Lambda)$-uniformly elliptic. That is, there exist $0<\lambda\leq\Lambda$ such that
\begin{equation}\label{eq_ellipticity}
	\lambda\left\|N\right\|\,\leq\,F(M\,+\,N)\,-\,F(M)\,\leq\,\Lambda\left\|N\right\|,
\end{equation}
for every $M,\,N\in\mathcal{S}(d)$ with $N\geq 0$. In addition, we suppose $F(0)=0$.
\end{Assumption}
By taking $M\equiv 0$ in A\ref{assump_F}, it follows that
\[
	\lambda\left\|N\right\|\,\leq\,F(N)\,\leq\,\Lambda\left\|N\right\|,
\]
for every $N\geq 0$. Therefore, uniform ellipticity implies that $F$ satisfies a coercivity condition over non-negative matrices. 

\begin{Assumption}[Convexity of the operator $F$]\label{assump_Fconvex}
We suppose the operator $F=F(M)$ to be convex with respect to $M$.
\end{Assumption}

In certain cases, we must impose further conditions on the growth regime of the operator $F$. Namely, extend \eqref{eq_ellipticity} to symmetric matrices $N\in\mathcal{S}(d)$. This is the content of the next assumption.
%\begin{Assumption}[Regularity of the operator $F$]\label{assump_FC1}
%We suppose $F$ to be continuously differentiable: $F_{ij}$ is continuous in $\mathcal{S}(d)$ for every $i,\,j=1,\dots,\,d$.
%\end{Assumption}

\begin{Assumption}\label{assump_Fmdqelliptic}
We suppose that $F$ satisfies
\[
	\lambda\left\|M\right\|\,\leq\,F(M)\,\leq\,\Lambda\left\|N\right\|,
\]
for every $M\in\mathcal{S}(d)$ and $F(0)=0$.
\end{Assumption}
Note that A\ref{assump_Fmdqelliptic} includes A\ref{assump_F}. Consequential on A\ref{assump_Fmdqelliptic} is a coercivity condition for $F$ in the entire space $\mathcal{S}(d)$. To compare A\ref{assump_F} and A\ref{assump_Fmdqelliptic} amounts to observe a change in the ellipticity cone of the operator. In fact, A\ref{assump_Fmdqelliptic} confines the image of the operator to a $(\lambda,\Lambda)$-cone contained in the upper quadrants of $\mathbb{R}^{d^2+1}$. See Figure \ref{fig_cone}.

 \bigskip

\begin{figure}[h!]
\centering

\psscalebox{.4 .4} % Change this value to rescale the drawing.
{
\begin{pspicture}(0,-6.3)(20.75,6.3)
\definecolor{colour0}{rgb}{0.8,0.8,0.8}
\pspolygon[linecolor=colour0, linewidth=0.04, fillstyle=solid,fillcolor=colour0](2.4,-5.5)(18.4,-0.7)(12.4,5.3)
\psline[linecolor=black, linewidth=0.04, arrowsize=0.05291667cm 2.0,arrowlength=1.4,arrowinset=0.0]{->}(2.4,-5.9)(2.4,5.7)
\psline[linecolor=black, linewidth=0.04, arrowsize=0.05291667cm 2.0,arrowlength=1.4,arrowinset=0.0]{->}(2.0,-5.5)(19.2,-5.5)
\psline[linecolor=colour0, linewidth=0.04](2.4,-5.5)(19.6,-0.3)(19.6,-0.3)
\psline[linecolor=colour0, linewidth=0.04](2.4,-5.5)(12.8,5.7)
\rput[bl](13.2,5.7){\Huge{$\Lambda$-line}}
\rput[bl](18.8,-1.5){\Huge{$\lambda$-line}}
\rput[bl](18.0,-6.3){\Huge{$M$}}
\rput[bl](0.0,4.5){\Huge{$F(M)$}}
\psbezier[linecolor=black, linewidth=0.04](2.4,-5.5)(2.4,-5.5)(4.305573,-3.9472136)(5.2,-3.5)(6.094427,-3.0527864)(5.851317,-3.4162278)(6.8,-3.1)(7.7486835,-2.7837722)(7.4636707,-2.2511234)(8.4,-1.9)(9.336329,-1.5488765)(9.463671,-2.2511234)(10.4,-1.9)(11.336329,-1.5488765)(11.6,-0.7)(11.6,-0.7)(11.6,-0.7)(12.4,0.5)(13.2,0.9)(14.0,1.3)(14.8,0.5)(16.4,0.1)
\rput[bl](15.6,-0.7){\Huge{$F_2$}}
\psbezier[linecolor=black, linewidth=0.04](2.4,-5.5)(2.4,-5.5)(13.2,-0.3)(13.2,4.1)
\rput[bl](12.0,3.3){\Huge{$F_1$}}
\end{pspicture}
}
\caption{{\bf Assumption A\ref{assump_Fmdqelliptic} -- }The $(\lambda,\Lambda)$-ellipticity condition, as in Assumption A\ref{assump_F}, confines the image of the operator to the area between the lines $\lambda\|\,\cdot\,\|$ and $\Lambda\|\,\cdot\,\|$. However, A\ref{assump_Fmdqelliptic} confines the image of the operator to the region between those lines \emph{within the upper hemisphere} of $\mathcal{S}(d)\times\mathbb{R}$. Here we can find two examples of operators satisfying A\ref{assump_Fmdqelliptic}; a convex operator, $F_1$, and a nonconvex one, $F_2$.}
\label{fig_cone}
\end{figure}
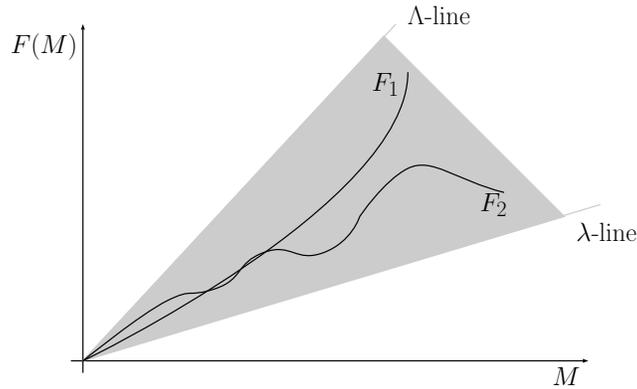

\bigskip

In the next section we collect a number of definitions and auxiliary results used throughout the paper. 

\subsection{Preliminary notions and results}\label{subsec_pnr}

When equipping \eqref{eq_main} with boundary conditions in the sense of Sobolev, we make use of an affine subspace of $W^{2,p}(B_1)$. 

\begin{Definition}[Sobolev spaces $W^{2,p}_g(B_1)$]\label{def_w2pg} 
Given $g\in W^{2,p}(B_1)$, we say that $u\in W^{2,p}_g(B_1)$ if $u-g\in W^{2,p}_0(B_1)$.
\end{Definition}
Once this definition is available, to prescribe $u=g$ on $\partial B_1$ in the Sobolev sense means to consider $u\in W^{2,p}_g(B_1)$. Instead of satisfying the boundary condition in the classical sense, we only require the difference between $u$ and the data to be an element of $W^{2,p}_0(B_1)$.

To properly explore the connection between \eqref{eq_mainvar} and \eqref{eq_main}, we introduce the definition of \emph{solution to the mean-field games system} in \eqref{eq_main}.

\begin{Definition}[Weak solution]A pair $(u,m)$ is said to be a weak solution to \eqref{eq_main} if the following conditions are satisfied:
\begin{enumerate}
	\item $u\in \mathcal{C}(B_1)$ and $m\in L^1(B_1)$ satisfies $m\geq 0$ in $B_1$;
	\item $u$ is a viscosity solution to 
	\[
		F(D^2u)\,=\,m^\frac{1}{p-1}\;\;\;\;\;\mbox{in}\;\;\;\;\;B_1;
	\]
	\item $m$ satisfies
	\[
		\int_{B_1}\left(F_{ij}(D^2u)\,m\right)\phi_{x_ix_j}\,\bd x\,=\,0,
	\]
	for every $\phi\in\mathcal{C}^\infty_c(B_1)$.
\end{enumerate}
\end{Definition}

In general, a solution to \eqref{eq_main} is a critical point for the functional \eqref{eq_mainvar}. Nevertheless, up to this point, we have not enough information to ensure that solutions to the MFG system are indeed minimizers of the functional $I[u]$. To distinguish solutions minimizing this functional, we introduce a definition.

\begin{Definition}[Minimizing solutions]\label{def_minimizingsolutions}
We say that $(u^*,m^*)$ is a minimizing solution if it solves \eqref{eq_main} and satisfies
\[
	I[u^*]\,\leq\,I[u]
\]
for every $u\in W^{2,p}_g(B_1)$.
\end{Definition}

As concerns the study of relaxation methods, the convex envelope of $F$ plays a central role. In what follows, we define this object.

\begin{Definition}[Convex envelope]\label{def_convexenv}
Let $F:\mathcal{S}(d)\to\mathbb{R}$ be a fully nonlinear operator satisfying A\ref{assump_F}. The convex envelope of $F$ is the operator $\Gamma_F:\mathcal{S}(d)\to\mathbb{R}$ defined by
\[
	\Gamma_F(M)\,:=\,\sup\left\lbrace G(M)\,|\,G\,\leq\,F\;\;\mbox{and}\;\;G\;\mbox{is convex}\right\rbrace
\]
\end{Definition}
%An interesting aspect of Definition \ref{def_convexenv} regards the ellipticity/coercivity of the convex envelope, provided $F$ satisfies either A\ref{assump_F} or A\ref{assump_Fmdqelliptic}. In fact, there can be found examples of nonconvex $(\lambda,\Lambda)$-elliptic operators whose convex envelope fail ellipticity; compare with Proposition \ref{prop_cecoercive}.

We are interested in gains of regularity produced by the MFG coupling \textit{vis-a-vis} the equations in \eqref{eq_main} taken in isolation. To pursue this direction, we rely on the gains of integrability for the double-divergence equation.

\begin{Lemma}[Gains of integrability for the double-divergence equation]\label{lem_fabesstroock}
Let $m\in L^1(B_1)$ be a non-negative weak solution to the second equation in \eqref{eq_main}. Suppose A\ref{assump_F} is in force. Then, $m\in L^\frac{d}{d-1}_{loc}(B_1)$ and there exists a constant $C>0$ such that
\[
	\left\|m\right\|_{L^\frac{d}{d-1}(B_{1/2})}\,\leq\,C.
\]
\end{Lemma}

For a proof of Lemma \ref{lem_fabesstroock} we refer the reader to \cite{fabesstroock}. See also \cite{bkrsbook} and the references therein. The importance of this result lies in the fact that an integrable solution to the double-divergence equation has higher integrability, depending explicitly on the dimension $d$.

Because our arguments are heavily based on the direct methods in the calculus of variations, part of our arguments are related to the weakly lower semi-continuity of the functional $I[u]$. An important ingredient in the study of this property is the following lemma.

\begin{Lemma}[Mazur Theorem]\label{lem_mazur}
Let $X$ be a linear space and $\ell:X\to\mathbb{R}^+$ be a norm defined on $X$. If $(x_n)_{n\in\mathbb{N}}\subset X$ is such that
\[
	x_n\,\rightharpoonup\,x\;\;\;\;\;\;\mbox{in}\;\;\;\;\;X,
\]
there exists a sequence $(y_m)_{m\in\mathbb{N}}\subset\co (x_n)_{n\in\mathbb{N}}$ such that:
\begin{enumerate}
\item for every $m\in\mathbb{N}$ there exists $M\in\mathbb{N}$ and $(\alpha_m^n)_{n=1}^M$ with 
\[
	\alpha_m^n\,>\,0\;\;\;\;\;\;\;\;\;\;\sum_{n=1}^M\alpha_m^n\,=\,1
\]
and
\[
	y_m\,=\,\sum_{n=1}^M\alpha_m^nx_n;
\]
\item in addition, we have
\[
	\ell\left(y_m\,-\,x\right)\,\longrightarrow \,0\;\;\;\;\;\;\;\;\;\;\mbox{as}\;\;\;\;\;\;\;\;\;\;m\to\infty.
\]
\end{enumerate}
\end{Lemma}
For a proof of this result, we refer the reader to \cite[p. 120, Theorem 2]{yosidafa}. See also \cite{rudinfa} and \cite{dacorognadm}. In the context of functions $u\in W^{2,p}_g(B_1)$, we make use of a Poincar\'e's inequality depending intrinsically on $g$. This is the content of the next lemma.

\begin{Lemma}[Poincar\'e inequality]\label{lem_poincare}
Let $u\in W^{1,p}_g(B_1)$ and $C_p>0$ be the Poincar\'e's constant associated with $L^p(B_1)$. Then, for every $C<C_p$ there exists $C_1(C,C_p)>0$ and $C_2\geq0$ such that
\[
	\int_{B_1}|Du|^p\bd x\,-\,C\int_{B_1}|u|^p\bd x\,+\,C_2\,\geq\, C_1\left(\int_{B_1}|Du|^p\bd x\,+\,\int_{B_1}|u|^p\bd x\right). 
\]
\end{Lemma}

For the proof of Lemma \ref{lem_poincare}, it suffices to apply the compactly supported version of the Poincar\'e's Inequality to the function $\tilde{u}:=u-g$. For the detailed argument, we refer the reader to \cite[Lemma 2.7, p. 22]{dalmaso}.  In the next section we explore the implications of the MFG coupling to the regularity of the solutions to \eqref{eq_main}.

\section{Gains of regularity through MFG couplings}\label{sec_gofreg}

In this section we explore the gains of regularity yielded by the coupling \eqref{eq_main}. This analysis is motivated by findings in the literature of mean-field games regarding the smoothness of solutions to MFG systems. In that context, the work-horse of the theory is the coupling of a Hamilton-Jacobi with a Fokker-Planck equation. 

When taken in isolation, those equations are solvable in the weak sense in regularity classes strictly below those required by classical solutions, However, in the presence of a suitable MFG coupling, smooth solutions are available. Motivated by the toy-models studied in the literature, in what follows we investigate gains of regularity for the solutions to fully nonlinear mean-field games. We focus on estimates in Sobolev spaces. 

Let $u\in \mathcal{C}(B_1)$ be a viscosity solution of
\[
	F(D^2u)\,=\,\mu^\frac{1}{p\,-\,1} \;\;\;\;\;\mbox{in}\;\;\;\;\;B_1,
\]
where $F$ satisfies A\ref{assump_F}-A\ref{assump_Fconvex}, $\mu\in L^1(B_1)$ is a non-negative function and $p-1\geq d$. It follows that $u\in W_{loc}^{2,p-1}(B_1)$, with  appropriate estimates. See \cite{caffarelli89}, \cite{escauriaza93} and \cite[Chapter 7]{ccbook}. We also refer the reader to \cite{pimtei}, where the convexity assumption on the operator is slightly weakened. Here we are interested in the case where the pair $(u,\mu)$ solves \eqref{eq_main}.

\begin{Proposition}[Improved regularity in Sobolev spaces]\label{prop_gains}
Let $(u,\mu)$ be a weak solution to \eqref{eq_main}. Suppose A\ref{assump_F}-A\ref{assump_Fconvex} are in force. Then, $u\in W^{2,\frac{d(p-1)}{d-1}}_{loc}(B_1)$. In addition, there exists $C>0$ such that 
\[
	\left\|u\right\|_{W^{2,\frac{d(p-1)}{d-1}}(B_{1/2})}\,\leq\,C\left(\left\|u\right\|_{L^\infty(B_1)}\,+\,\left\|\mu\right\|_{L^1(B_1)}^\frac{1}{p-1}\right).
\]
\end{Proposition}
\begin{proof}
The result follows by combining standard arguments in $W^{2,p}$-regularity theory with Lemma \ref{lem_fabesstroock}. Indeed, under A\ref{assump_F}-A\ref{assump_Fconvex}, it remains to verify that the right-hand side of the first equation in \eqref{eq_main} is bounded in $L^\frac{d(p-1)}{d-1}(B_1)$. 

Because of Lemma \ref{lem_fabesstroock}, we have $\mu\in L^\frac{d}{d-1}(B_1)$. Therefore, there exists $C>0$ such that
\[
	\int_{B_1}\left|\mu^\frac{1}{p-1}\right|^{\frac{d(p-1)}{d-1}}\,\bd x\,\leq\,C
\]
and the result follows.
\end{proof}

\begin{figure}[h!]
\centering

\psscalebox{.4 .4} % Change this value to rescale the drawing.
{

\begin{pspicture}(0,-5.9534793)(19.906958,5.9534793)
\definecolor{colour0}{rgb}{0.8,0.8,0.8}
\psline[linecolor=black, linewidth=0.04, arrowsize=0.05291667cm 2.0,arrowlength=1.4,arrowinset=0.0]{->}(3.2,-5.553479)(3.2,6.0465207)
\psline[linecolor=black, linewidth=0.04, arrowsize=0.05291667cm 2.0,arrowlength=1.4,arrowinset=0.0]{->}(2.8,-5.153479)(20.0,-5.153479)
\rput[bl](19.2,-5.9534793){\Huge{$d$}}
\rput[bl](0.0,3.2465208){\Huge{$2(p-1)$}}
\rput[bl](6.0,-5.9534793){\Huge{$2$}}
\psline[linecolor=colour0, linewidth=0.04, linestyle=dashed, dash=0.17638889cm 0.10583334cm](6.0,-5.153479)(6.0,5.246521)(6.0,5.246521)
\psline[linecolor=colour0, linewidth=0.04, linestyle=dashed, dash=0.17638889cm 0.10583334cm](3.2,-2.7534792)(19.2,-2.7534792)
\psline[linecolor=colour0, linewidth=0.04, linestyle=dashed, dash=0.17638889cm 0.10583334cm](3.2,3.6465209)(19.2,3.6465209)
\psbezier[linecolor=black, linewidth=0.04](6.0,3.6465209)(6.4,-0.7534792)(11.6,-1.9534792)(19.2,-2.3534791564941404)
\rput[bl](0.4,-3.153479){\Huge{$(p-1)$}}
\rput[bl](2.4,5.246521){\Huge{$q$}}
\psdots[linecolor=black, dotstyle=|, dotsize=0.2](6.0,-5.153479)
\end{pspicture}
}
\caption{Gains of regularity. When equipped with a non-negative right-hand side $\mu\in L^1(B_1)$, the equation $F=\mu^{1/(p-1)}$ has solutions in $W^{2,q}_{loc}(B_1)$, for $d-\varepsilon<q\leq p-1$. Nevertheless, if we consider a pair $(u,\mu)$, solutions to \eqref{eq_main}, the regularity of the solutions to $F=\mu^{1/(p-1)}$ improves; in that case, the range for $d-\varepsilon<q<d(p-1)/(d-1)$ increases by a dimension-dependent factor.}
\label{fig_gains}
\end{figure}
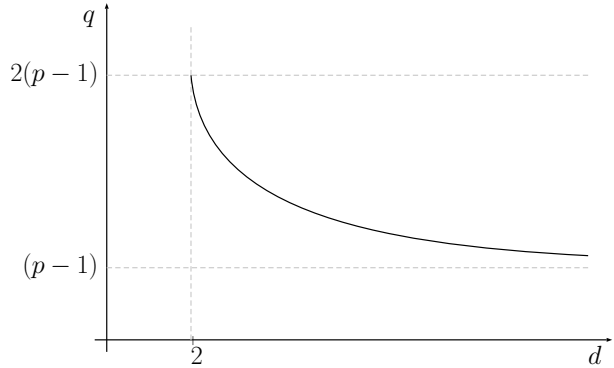

The conclusion of Proposition \ref{prop_gains} is that the MFG structure in \eqref{eq_main} entails improved regularity levels for $u$ in Sobolev spaces. A straightforward consequence of this fact is put forward in the next corollary.

\begin{Corollary}[Improved regularity in H\"older spaces]\label{cor_holder}
Let $(u,\mu)$ be a weak solution to \eqref{eq_main} and suppose A\ref{assump_F}-A\ref{assump_Fconvex} are in force. Suppose further that $p-1>d$. Then, $u\in \mathcal{C}^{1,\alpha^*}_{loc}(B_1)$, for
\[
	\alpha^*\,:=\,1\,-\,\frac{d\,-\,1}{p\,-\,1}.
\]
In addition, there exists $C>0$ such that 
\[
	\left\|u\right\|_{\mathcal{C}^{1,\alpha^*}(B_{1/2})}\,\leq\,C\left(\left\|u\right\|_{L^\infty(B_1)}\,+\,\left\|\mu\right\|_{L^1(B_1)}^\frac{1}{p-1}\right).
\]
\end{Corollary}
\begin{Remark}\normalfont
It is clear that the results in this section hold in the context of an operator $F=F(M,x)$ with variable coefficients, provided it satisfies an oscillation estimate with respect to its fixed-coefficients counterpart. We refer the reader to \cite{ccbook}. In addition, the convexity assumption on $F$ might be weakened. Indeed, we only require a limiting profile with $\mathcal{C}^{1,1}$-estimates to be available; such a limiting configuration can be given by the fixed-coefficients counterpart of $F(M,x)$ or its \emph{recession operator}. See \cite{pimtei} and \cite{silvtei}.
\end{Remark}

\begin{Remark}[Non-monotone couplings]\normalfont
We notice the gains of regularity discussed here would be available also under the non-monotone coupling term $z\mapsto -z^\frac{1}{p-1}$, instead of the monotone one in \eqref{eq_main}. 
\end{Remark}

In the next section we make use of variational techniques to investigate the existence of solutions to \eqref{eq_main}.

\section{Existence of solutions}\label{sec_convex}

Here we investigate the existence of minimizers to \eqref{eq_mainvar} as well as the existence of solutions to \eqref{eq_main}. We work both under A\ref{assump_F} and A\ref{assump_Fmdqelliptic}. In the context of the former, we prove the existence of minimizers for \eqref{eq_mainvar} in the class of convex functions $\mathcal{A}$ defined as
\[
	\mathcal{A}\,:=\,\left\lbrace u\,\in\,W^{2,p}_g(B_1),\;\; u\;\;\mbox{convex}\right\rbrace,
\]
where $g\in W^{2,p}(B_1)$ is convex. When working under A\ref{assump_Fmdqelliptic} we establish the existence of a minimizer for \eqref{eq_mainvar} in $W^{2,p}_g(B_1)$. In addition, we prove that such a minimizer is a solution to the associated Euler-Lagrange equation, which leads to the existence of a solution to \eqref{eq_main}.

\begin{Proposition}[Existence of minimizers]\label{teo_existconvex}
Suppose A\ref{assump_F}-A\ref{assump_Fconvex} are in force. Then, there exists $u^*\in\mathcal{A}$ such that 
\[
	I[u^*]\,\leq\,I[u]\;\;\;\;\;\;\;\;\mbox{for every}\;\;\;\;\;\;\;\;u\in\,\mathcal{A}.
\]
\end{Proposition}

The proof of Proposition \ref{teo_existconvex} relies on the weak lower semicontinuity of the functional $I$ in the class $\mathcal{A}$. This is the content of the next proposition.

\begin{Proposition}\label{prop_weaklsc}
Suppose A\ref{assump_F}-A\ref{assump_Fconvex} hold true. Let $(u_n)_{n\in\mathbb{N}}\in\mathcal{A}$ be such that 
\[
	D^2u_n\,\rightharpoonup \, D^2u_\infty\;\;\;\;\;\mbox{in}\;\;\;\;\;L^p(B_1,\mathbb{R}^{d^2}).
\]
Then, 
\[
	\liminf_{n\to\infty} I[u_n]\,\geq\,I[u_\infty].
\]
\end{Proposition}
\begin{proof}We start by establishing the strong lower-semicontinuity of the functional $I$. Then, we resort to Lemma \ref{lem_mazur} to obtain the weak lower-semicontinuity. 

\medskip

\noindent{\bf Step 1}

\medskip

Because $(u_n)_{n\in\mathbb{N}}\subset\mathcal{A}$, we have $D^2u_n(x)\geq 0$ almost everywhere in $B_1$. Therefore, A\ref{assump_F} yields
\begin{equation}\label{eq_lowerF}
	F(D^2u_n(x))\,\geq\,\lambda\|D^2u_n(x)\|\,\geq\,0,\;\;\;\;\;\mbox{a.e.}\,x\,\in\,B_1.
\end{equation}
Now, suppose $D^2u_n\to D^2u_\infty$ strongly in $L^p(B_1)$; through a subsequence, if necessary. We infer that $D^2u_n\to D^2u_\infty$ almost everywhere in $B_1$. Hence, Fatou's Lemma builds upon the (Lipschitz) continuity of $F$ to imply
\[
	\int_{B_1}\left[F(D^2u_\infty(x)\right]^p\bd x\,\leq\, \liminf_{n\to\infty}\int_{B_1}\left[F(D^2u_n(x)\right]^p\bd x.
\]

\medskip

\noindent{\bf Step 2}

\medskip

To prove the weak lower-semicontinuity we make use of Lemma \ref{lem_mazur}. Notice that
\[
	0\,\leq \,\liminf_{n\to\infty}I[u_n]\,<\,\infty.
\]
For every $\varepsilon>0$, there exists $n_\varepsilon\in\mathbb{N}$ such that for $n>n_\varepsilon$, we have
\begin{equation}\label{eq_infepsilon}
	I[u_n]\,\leq\,\liminf_{n\to\infty}I[u_n]\,+\,\varepsilon.
\end{equation}
For $\varepsilon>0$ fixed, Lemma \ref{lem_mazur} ensures the existence of $(D^2v_m)_{m\in\mathbb{N}}\subset\co(D^2u_n)_{n\in\mathbb{N}}$ such that $D^2v_m\to D^2u_\infty$ in $L^p(B_1)$. Moreover, for every $m\in\mathbb{N}$, one can find $M>n_\varepsilon$ and points $\alpha_m:=(\alpha_m^{n_\varepsilon},\,\dots,\alpha_m^M)$ in the $(M-n_\varepsilon+1)$-dimensional simplex such that 
\begin{equation}\label{eq_conv1}
	D^2v_m\,=\,\sum_{i=n_\varepsilon}^M\alpha_m^iD^2u_i.
\end{equation}
The convexity of $[F(\,\cdot\,)]^p$ combined with the strong lower-semicontinuity and \eqref{eq_infepsilon} leads to
\[
	I[u_\infty]\,\leq\,I[v_m]\,\leq\,\sum_{i=n_\varepsilon}^M\alpha_m^iI[u_i]\,\leq\,\liminf_{n\to\infty}I[u_n]\,+\,\varepsilon.
\]
Since $\varepsilon>0$ was taken arbitrarily, the proof is complete.
\end{proof}

Once the weak lower-semicontinuity of the functional $I$ has been established, we prove the existence of a minimizer to \eqref{eq_mainvar} in the class $\mathcal{A}$.

\begin{proof}[Proof of Proposition \ref{teo_existconvex}]
We start by setting
\[
	\underline{m}\,:=\,\inf_{u\in\mathcal{A}}I[u].
\]
Clearly, $\underline{m}> 0$, since $g$ is not trivial. On the other hand, $\underline{m}\leq I[g]<+\infty$. Hence, $0< \underline{m}<\infty$. Let $(u_n)_{n\in\mathbb{N}}\subset\mathcal{A}$ be a minimizing sequence. There exists $N\in\mathbb{N}$ such that 
\[
	I[u_n]\,\leq\,\underline{m}\,+\,1\;\;\;\;\;\;\mbox{for every}\;\;\;\;\;\;n\,>\,N.
\]
Therefore,
\begin{equation}\label{eq_lowerF2}
	\|D^2u_n\|_{L^p(B_1)}\,\leq\,\frac{1}{\lambda^p}\int_{B_1}\left[F(D^2u_n)\right]^p\bd x\,\leq\,C(\underline{m},\lambda,p),
\end{equation}
for every $n>N$. As a result, we infer that $(D^2u_n)_{n\in\mathbb{N}}$ is uniformly bounded in $L^p(B_1)$. Because of Lemma \ref{lem_poincare}, we conclude that $(u_n)_{n\in\mathbb{N}}$ is uniformly bounded in $W^{2,p}_g(B_1)$. Therefore, there exists $u_\infty\in\mathcal{A}$ such that $u_n\rightharpoonup u_\infty$ in $W^{2,p}_g(B_1)$.

Proposition \ref{prop_weaklsc} implies that 
\[
	I[u_\infty]\,\leq\,\liminf_{n\to\infty}I[u_n]\,=\,\underline{m};
\]
we set $u^*\equiv u_\infty$ and the proof is complete.
\end{proof}

\begin{Remark}\label{rem_a4}\normalfont
If we replace A\ref{assump_F} with A\ref{assump_Fmdqelliptic} the conclusion of the Theorem \ref{teo_existconvex} changes. In fact, under A\ref{assump_Fmdqelliptic}, the inequality in \eqref{eq_lowerF} holds true for every $u\in W^{2,p}_ g(B_1)$. As a result, we obtain that $I$ is weakly lower semicontinuous over $W^{2,p}_g(B_1)$. Moreover, A\ref{assump_Fmdqelliptic} yields \eqref{eq_lowerF2} for every minimizing sequence $(u_n)_{n\in\mathbb{N}}\subset W^{2,p}_g(B_1).$

The conclusion is that under A\ref{assump_Fconvex} and A\ref{assump_Fmdqelliptic}, there exists $u^*\in W^{2,p}_g(B_1)$ such that 
\[
	I[u^*]\,\leq\,I[u]\;\;\;\;\;\;\;\mbox{for every}\;\;\;\;\;\;u\,\in\,W^{2,p}_g(B_1).
\]
That is, problem \eqref{eq_mainvar} admits a minimizer in $W^{2,p}_g(B_1)$.
\end{Remark}

In the sequel, the discussion in Remark \ref{rem_a4} builds upon standard methods in calculus of variations to produce information on the Euler-Lagrange equation \eqref{eq_main}.

\begin{teo}\label{teo_eulerlagrange}
Suppose A\ref{assump_Fconvex} and A\ref{assump_Fmdqelliptic} hold true. Then, there exists a minimizer $u^*\in W^{2,p}_g(B_1)$ for \eqref{eq_mainvar} in the space $W^{2,p}_g(B_1)$. In addition, such a minimizer yields a solution $(u^*,m^*)$, solutions to the fully nonlinear mean-field game system \eqref{eq_main}. 
\end{teo}
\begin{proof}
We start by noticing that, because $F$ is $\mathcal{L}$-Lipschitz-continuous in $\mathcal{S}(d)$ and $F(0)=0$, we infer that
\[
	\left|F(M)\right|^p\,\leq\,\mathcal{L}^p\left|M\right|^p
\]
and
\[
	\left|F_{ij}(M)\right|\,\leq\,\mathcal{L},
\]
uniformly in $M$. 
By combining Theorem \ref{teo_existconvex} with with Remark \ref{rem_a4}, we infer the existence of a minimizer $u^*$ for \eqref{eq_mainvar} in the Sobolev space $W^{2,p}_g(B_1)$. Therefore, for every $\phi\in\mathcal{C}^\infty_0(B_1)$, we have
\begin{align}\nonumber
	\int_{B_1}F_{ij}(D^2u^*)F(D^2u^*)^{p-1}\phi_{x_ix_j}\bd x\,&\leq\,\int_{B_1}\left|F_{ij}(D^2u^*)\right|\left|F(D^2u^*)\right|^{p-1}\left|\phi_{x_ix_j}\right|\bd x\\
		&\leq\,C\left(1\,+\,\|D^2u^*\|_{L^p(B_1)}^p\right),
\end{align}
where $C=C(\lambda,\Lambda,d,p)$. Since $u^*\in W^{2,p}_g(B_1)$, the weak form of the Euler-Lagrange equation associated with \eqref{eq_mainvar} is well-defined for every $\phi\in\mathcal{C}^\infty_c(B_1)$.

By setting $m^*:=F(D^2u^*)$, we notice that  $u^*$ is a viscosity solution to the first equation in \eqref{eq_main}. In addition, A\ref{assump_Fmdqelliptic} implies that $m^*\geq 0$. Finally, $m^*$ is clearly integrable, since
\[	
	\int_{B_1}m^*\,\bd x\,\leq\,C(p,d,\Lambda)\,+\,\int_{B_1}\left\|D^2u^*(x)\right\|^p\,\bd x\,\leq\,C(p,d,\Lambda,g).
\]
The function $m^*$ satisfies the second equation in \eqref{eq_main} by construction. Therefore, the proof is complete.
\end{proof}
\begin{Remark}[Logarithmic nonlinearities]\normalfont
Under appropriate convexity assumptions on $F$, it is reasonable to expect that minor modifications of our arguments would account for a functional of the form
\begin{equation}\label{eq_mainvarlog}
	\tilde{I}[u]\,:=\,\int_{B_1}e^{F(D^2u)}\,\bd x\;\;\;\;\;\longrightarrow\;\;\;\;\;\min.
\end{equation}
In fact, if $F=F(M)$ is convex, the function $e^{F(M)}$ is also convex. Also, we notice that under A\ref{assump_Fmdqelliptic} 
\[
	\lambda\|M\|\,\leq\,e^{\lambda\|M\|}\,\leq\,e^{F(M)}.
\]
The interest in \eqref{eq_mainvarlog} is mostly motivated by its Euler-Lagrange equation. It gives rise to the following MFG system:
\begin{spacing}{1.0} 
\begin{equation}\label{eq_mainlog}
	\begin{cases}
		F(D^2u)\,=\,\ln m&\;\;\;\;\;\mbox{in}\;\;\;\;\;B_1\\
		\left(F_{ij}(D^2u)\,m\right)_{x_ix_j}\,=\,0&\;\;\;\;\;\mbox{in}\;\;\;\;\;B_1.
	\end{cases}
\end{equation}
\end{spacing}
\vspace{.09in}
The problem in \eqref{eq_mainlog} is known as \emph{MFG with logarithmic nonlinearities} and plays an important role in the mean-field games literature.

\end{Remark}
\begin{Remark}[Lower order terms]\label{rem_lowerorder}\normalfont
We notice that Theorems \ref{teo_existconvex} and \ref{teo_eulerlagrange} can be adapted to include more general operators, namely, depending on lower-order terms. Suppose
\[
	F:\mathcal{S}(d)\times\mathbb{R}^d\times\mathbb{R}\times B_1\longrightarrow\mathbb{R}
\]
is such that, for every $M\in\mathcal{S}(d)$, $\xi\in\mathbb{R}^d$, $r\in\mathbb{R}$ and $x\in B_1$, we have
\[
	\lambda\|M\|+\alpha\|\xi\|+\beta|r|\leq F(M,\xi,r,x)\leq\,\Lambda\|M\|+A\|\xi\|+B|r|,
\]
for some $0<\lambda\leq \Lambda$, $0<\alpha\leq A$ and $0<\beta\leq B$. Here, the Euler-Lagrange equation produces the more general MFG system
\begin{spacing}{1.0} 
\begin{equation*}\label{eq_mainlot}
	\begin{cases}
		F(D^2u,Du,u,x)\,=\,m^{1/(p-1)}&\;\;\;\;\;\mbox{in}\;\;\;\;\;B_1\\
		\left(F_{ij}\,m\right)_{x_ix_j}\,+\,\left(F_{ij}\,m\right)_{x_i}\,+\,F_{ij}\,m\,=\,0&\;\;\;\;\;\mbox{in}\;\;\;\;\;B_1.
	\end{cases}
\end{equation*}
\end{spacing}
\vspace{.09in}

As mentioned in the introduction, we also believe our methods and techniques would extend to operator of the form
\[
	F(M)\,+\,H(p),
\]
provided $H$ satisfies \eqref{eq_chatopracaralho}. For the sake of presentation, we refrain from pursuing explicitly those computations. It is clear that \eqref{eq_chatopracaralho} builds upon our previous computations to produce the required coercivity and weak-lower semicontinuity, with eventual conditions on the exponent $q$. We finish this remark by observing that different choices of $q$ would lead to equations involving different operators. 

\end{Remark}

In the sequel we investigate the minimizers of \eqref{eq_mainvar} in the cases where the operator $F$ fails to be convex.

\section{Analysis in the non-convex setting}\label{sec_relaxation}

In what follows we consider the case of functionals driven by non-convex operators $F$. As it is well known, one cannot ensure the existence of minimizers in this context. Nonetheless, information about 
\[
	\min_{u\,\in \,W^{2,p}_g(B_1)}\,I[u]
\]
still can be obtained through the so-called \emph{relaxation methods}, that is, by the study of the relaxed problem
\begin{equation}\label{eq_mainvarce}
	\bar{I}[u]\,:=\,\int_{B_1}\,\left[\Gamma_F(D^2u)\right]^p\,\bd x\;\;\longrightarrow \;\;\min
\end{equation}
Here, we develop this approach and combine it with the discussion in Section \ref{sec_convex}. We start with a proposition.

\begin{Proposition}[Coercivity of the convex envelope]\label{prop_cecoercive}
Suppose A\ref{assump_Fmdqelliptic} is in force. Then, $\Gamma_F$ satisfies a coercivity condition of the form
\[
	\lambda\|M\|\,\leq\,\Gamma_F(M),
\]
for every $M\in\mathcal{S}(d)$.
\end{Proposition}
\begin{proof}
The result follows from the convexity of the norm. Since $\Gamma_F(M)$ is the supremum taken among all the convex functions $G:\mathcal{S}(d)\to\mathbb{R}$ evaluated at $M$, we must have
\[
	\lambda\|M\|\,\leq\,\sup\left\lbrace G(M)\,|\,G\,\leq\,F\;\;\mbox{and}\;\;G\;\mbox{is convex}\right\rbrace.
\]
\end{proof}

Now we are in position to state and prove the main result in this section. 

\begin{teo}\label{teo_relaxed}
Suppose A\ref{assump_Fmdqelliptic} is in force. Then, there exists $(u_n)_{n\in\mathbb{N}}\subset W^{2,p}_g(B_1)$ such that
\[
	I[u_n]\,\longrightarrow \,\bar{I}[u^*]
\]
and
\[
	u_n\,\rightharpoonup u^*\;\;\;\;\;\;\;\;\mbox{in}\;\;\;\;\;\;\;\;W^{2,p}(B_1)
\]
as $n\to\infty$, where $u^*\in W^{2,p}_g(B_1)$ is the minimizer of \eqref{eq_mainvarce}. 
\end{teo}
\begin{proof}
We start by showing that the convex envelope of $F$ inherits the growth regime imposed by A\ref{assump_Fmdqelliptic}. In fact, Proposition \ref{prop_cecoercive} yields the required lower bounds. To obtain the upper bound, notice that $\Gamma_F(M)$ is below $F$; hence
\[
	\Gamma_F(M)\,\leq\,\Lambda\|M\|,
\]
for every $M\in\mathcal{S}(d)$. Therefore, $\Gamma_F$ is a convex mapping satisfying A\ref{assump_Fmdqelliptic} and falls within the scope of Proposition \ref{teo_existconvex}. We conclude that \eqref{eq_mainvarce} has a minimizer $u^*$. That is, there exists $u^*\in W^{2,p}_g(B_1)$ such that
\[
	\bar{I}[u^*]\,\leq\,\bar{I}[u],
\]
for every $u\in W^{2,p}_g(B_1)$. Finally, we evoke standard relaxation results and the proof is complete.
\end{proof}

The contribution of Theorem \ref{teo_relaxed} is to provide information on problem \eqref{eq_mainvar} in the absence of convexity. Although it might fail to have a minimizer, we characterize its infimum in terms of the convex envelope of $F$. In addition, the element $u^*\in W^{2,p}_g(B_1)$ that produces the infimum of the relaxed problem is the weak limit of a sequence for which the original functional is defined.   

\section{Explicit solutions in dimension $d=1$}\label{sec_unid}

In this section we work out in detail an explicit example in dimension $d=1$. This concrete model allows us to produce explicit solutions with very simple structure and perform some analysis of the various features involved in the problem. Our analysis is motivated by \cite{dgomesuni}.

Here, we consider the open interval $(0,1)$ and especialize $F=F(z)$ to be given by
\begin{equation}\label{eq_explicit1}
	F(z)\,:=\,\left(1\,+\,z^p\right)^\frac{1}{p},
\end{equation}
where $p\geq 2$ is a fixed integer. Here we have 
\[
	F''(z)\,=\,(1\,-\,p)z^{2p-2}(1\,+\,z^p)^\frac{1-2p}{p}\,+\,(p\,-\,1)x^{p-1}(1\,+\,x^p)^\frac{1-p}{p};
\]
therefore, $F''(z)$ might be negative depending on the values of $z$ and $p$. In that case, $F$ fails to be convex with respect to $z$. In addition, if $p\in2\mathbb{N}+1$ is an odd number, the operator also lacks differentiability.

Under \eqref{eq_explicit1}, the functional in \eqref{eq_mainvar} becomes
\[
	I[u]\,=\,\int_0^11\,+\,u_{xx}^p(x)\,\bd x,
\]
whereas \eqref{eq_main} comprises
\begin{equation}\label{eq_eqexpl1}
	\left(1\,+\,u_{xx}^p(x)\right)^\frac{1}{p}\,=\,m^\frac{1}{p-1}
\end{equation}
and
\begin{equation}\label{eq_eqexpl2}
	\left(\frac{u_{xx}^{p-1}(x)\,m(x)}{\left(1\,+\,u_{xx}^p(x)\right)^\frac{p-1}{p}}\right)_{xx}\,=\,0.
\end{equation}
%It follows that 
%\[
%	\left(u^{p-1}_{xx}\right)_{xx}\,=\,0
%\]
%and, a
As a consequence, it follows that
\[
	\left(u^{p-1}_{xx}\right)_{xx}\,=\,0.
\] 
Thus, we discover
\begin{equation}\label{eq_exp1}
	u(x)\,=\,\frac{A\left((p-1)x\,-\,B\right)^{2+\frac{1}{p-1}}}{p(2p\,-\,1)}\,+\,Cx\,+\,D
\end{equation}
where $A$, $B$, $C$ and $D$ are constants. Using \eqref{eq_eqexpl1} we find
\begin{equation}\label{eq_exp2}
	m(x)\,=\,\left(1\,+\,\left[A\left((p\,-\,1)x\,-\,B\right)^\frac{1}{p-1}\right]^p\right)^\frac{1}{p},
\end{equation}

Now we turn our attention to the minimization problem driven by $I[u]$. We infer from \eqref{eq_exp1} that
\[
	I[u]\,=\,1\,+\,\frac{A\left(\left[(p\,-\,1)\,-\,B\right]^\frac{p}{p-1}\,-\,(-B)^\frac{p}{p-1}\right)}{p}
\]
Our goal is to characterize $A$ and $B$ in order to minimize $I[u]$. By resorting to the first order conditions, we find that such constants must be chosen in order to satisfy both
\[
	\left(A^2 (-A B)^{\frac{1}{p-1}}\right)^p-\left(A^2 (A (-B+p-1))^{\frac{1}{p-1}}\right)^p=0
\]
and 
\[
	A^{2p-1}\left[B p \left((-A B)^{\frac{1}{p-1}}\right)^p+p (-B+p-1) \left((A (-B+p-1))^{\frac{1}{p-1}}\right)\right]^p=0.
\]
To ensure that both constraints derived from the first order conditions are met, we must have $A\equiv 0$. Therefore, a solution to \eqref{eq_eqexpl1}-\eqref{eq_eqexpl2} minimizes the associated functional $I[u]$ if it takes the form
\[
	u(x)\,:=\,Cx\,+\,D\;\;\;\;\;\mbox{and}\;\;\;\;\;m(x)\,\equiv\,1.
\]
Among the solutions to the mean-field games system, those minimizing the functional comprise an affine mapping $u$ and a uniform distribution $m$.  A remarkable aspect of this toy-model is that minimizing solutions are independent of the power $p\geq 2$. 

In the presence of a boundary condition $u(0)=g(0)$ and $u(1)=g(1)$, we have
\[
	D\,:=\, g(0)\;\;\;\;\;\mbox{and}\;\;\;\;\;C\,:=\,g(1)\,-\,g(0).
\]
We sumarize our findings in the following theorem:

\begin{teo}[Explicit solutions]\label{teo_unidimensional}
Let $F$ be given as in \eqref{eq_explicit1}. Then
\begin{enumerate}
\item A solution $(u,m)$ to the associated mean-field games system is given by
\begin{equation*}
	u(x)\,=\,\frac{A\left((p-1)x\,-\,B\right)^{2+\frac{1}{p-1}}}{p(2p\,-\,1)}\,+\,Cx\,+\,D
\end{equation*}
and
\begin{equation*}
	m(x)\,=\,\left(1\,+\,\left[A\left((p\,-\,1)x\,-\,B\right)^\frac{1}{p-1}\right]^p\right)^\frac{1}{p},
\end{equation*}
where $A$, $B$, $C$, and $D$ are real constants. 

\item A minimizing solution $(u^*,m^*)$ comprises an affine mapping and a uniform distribution; i.e.,
\[
	u^*(x)\,=\,A^*x\,+\,B^*
\]
and
\[
	m^*(x)\,=\,1,
\]
where $A^*$ and $B^*$ are real constants.
\end{enumerate}
\end{teo}

\medskip

\begin{Remark}\normalfont
In \cite{gomes_ferreira}, the authors study the well-posedness of mean-field games systems through monotonicity methods. As a remark, they mention the case of fully nonlinear MFG; see \cite[Section 7.1]{gomes_ferreira}. In that paper the authors suppose the operators to be convex and of class $\mathcal{C}^\infty$ with respect to the Hessian of the solutions. Therefore, their analysis does not include the class of examples addressed in the present section. In fact, the operator
\[
	F(z)\,:=\,(1\,+\,z^p)^\frac{1}{p}
\]
fails to be convex for odd values of $p\in2\mathbb{N}+1$. In addition, $z\mapsto F(z)$ is not smooth.
This fact reinforces the importance of explicit examples, such as the one in \eqref{eq_eqexpl1}-\eqref{eq_eqexpl2}, accompanying results stated in more general settings. 
%By lacking convexity and differentiability, the class of examples generated in \eqref{eq_explicit1} with $p\in 2\mathbb{N}+1$ fall of the scope of the developments reported in \cite[Section 7.1]{gomes_ferreira}.
\end{Remark}

\bibliographystyle{plain}
\bibliography{biblio}

\bigskip

\noindent\textsc{P\^edra D. S. Andrade}\\
Department of Mathematics\\
Pontifical Catholic University of Rio de Janeiro -- PUC-Rio\\
22451-900, G\'avea, Rio de Janeiro-RJ, Brazil\\
\noindent\texttt{pedra.andrade@mat.puc-rio.br}

\bigskip

\noindent\textsc{Edgard A. Pimentel (Corresponding Author)}\\
Department of Mathematics\\
Pontifical Catholic University of Rio de Janeiro -- PUC-Rio\\
22451-900, G\'avea, Rio de Janeiro-RJ, Brazil\\
\noindent\texttt{pimentel@puc-rio.br}

%\Addresses

\end{document}